\documentclass[12pt]{amsart}
\usepackage[american]{babel}
\usepackage{amsmath}
\usepackage{latexsym}
\usepackage{amssymb}
\usepackage{enumerate}
\usepackage{tikz}
\usetikzlibrary{matrix,arrows,decorations.pathmorphing}
\usepackage{tikz-cd}
\theoremstyle{plain}
\newtheorem{Thm}[subsection]{Theorem}
\newtheorem{Cor}[subsection]{Corollary}
\newtheorem{Prop}[subsection]{Proposition}
\newtheorem{Lem}[subsection]{Lemma}
\newtheorem{KLem}[subsection]{Key Lemma}

\theoremstyle{definition}
\newtheorem{Ex}[subsection]{Example}

\newtheorem{Prob}[subsection]{Problem}
\renewcommand{\phi}{\varphi}
\newcommand{\RR}{\mathbb{R}}

\newcommand{\id}{\mathrm{id}}
\newcommand{\eps}{\varepsilon}
\renewcommand{\emptyset}{\varnothing}
\renewcommand{\setminus}{-}

\begin{document}

\title{{On small abstract quotients of Lie groups and locally compact groups}}
\author{Linus Kramer}
\address{Linus Kramer\newline\indent
Mathematisches Institut, Universit\"at M\"unster
\newline\indent
Einsteinstr. 62, 
48149 M\"unster,
Germany}
\email{linus.kramer{@}uni-muenster.de}
\thanks{Supported by SFB 878}
\dedicatory{Dem Andenken von G\"unter Pickert gewidmet.}
\maketitle

\section{Introduction}

Suppose that $G$ is a locally compact group, that $\Gamma$ is a discrete,
finitely generated group, and that 
\[
 \phi:G \longrightarrow \Gamma
\]
is an 'abstract' surjective homomorphism. We are interested in conditions
which imply that $\phi$ is automatically continuous. We obtain
a complete answer to this question in the case where $G$ is a topologically finitely generated
locally compact abelian group
or an almost connected Lie group. In these two cases the well-known structure theory
for such groups $G$ leads quickly to a solution. The question becomes much more
difficult if one assumes only that $G$ is a locally compact group.
This leads to interesting questions about normal subgroups in infinite products
and in ultraproducts. \L os' Theorem, the solution of the 5th
Hilbert Problem, and recent results by Nikolov--Segal can be combined to answer 
the question.

\smallskip
The topic of this article is thus of a rather group-theoretic nature. However,
locally compact transformation groups are ubiquitous in geometry.
Conversely, several proofs in this article rely on geometry, at different levels of
abstraction. The 
'bootstrap Lemma'~\ref{ElementaryLemma}
is based on a little exercise in spherical geometry, which
was still taught at German high schools in G\"unter Pickert's early years.
But we also use some observations about euclidean and spherical buildings.

\smallskip
The article is organized as follows. In Section~2 we consider small abstract quotients of
locally compact abelian groups. In Section~3 we do the same for Lie groups.
Section~4 then extends these results to compact and to locally compact groups.
In Section 5~we make some observations about normal subgroups in infinite products.
This leads ultimately to questions about normal subgroups in ultraproducts.
We adopt a sheaf-theoretic viewpoint which is elementary but useful.
Section 6~gives a detailed proof for the nonexistence of countable
quotients of compact connected perfect groups, following Nikolov--Segal.

The necessary background material on Lie groups and locally compact groups can be found in the 
excellent books by Hilgert--Neeb \cite{HilgertNeeb}, Stroppel \cite{Stroppel}, and Hofmann--Morris \cite{HMCompact}, \cite{HMPro}.
As the referee pointed out, it is likely that many of our results can be extended to Pro-Lie groups, using the structure 
theory developed in \cite{HMPro}. At some places, this is straight-forward, as in Corollary \ref{ProLieAbelian},
or in Lemma \ref{FiniteIndexLieLemma}, which holds verbatim for Pro-Lie groups.
The generalization of other results in this article would seem to require new  ideas.

\subsection{Notation and conventions} 
The cardinality of a set $S$ is written as $\#S$.
We call a set $S$ is \emph{countable} if $\#S\leq\aleph_0$.
Throughout this article, all topological groups and spaces are assumed to be Hausdorff,
unless stated otherwise.
The identity component in a topological group $G$ is denoted by $G^\circ$.
This is a closed normal subgroup of $G$. We say that a topological group $G$ is \emph{topologically
finitely generated} if $G$ has a finitely generated dense subgroup.

We call a quotient $G/N$ of a topological group $G$ \emph{abstract} if no topological assumptions 
like closedness of the subgroup $N$ in $G$ are made. Similarly, we call a homomorphism between topological groups abstract if
no continuity assumptions are made.

We frequently use the fact that a subgroup of finite index in a finitely
generated group is also finitely generated. It suffices to prove this for finitely generated
free groups, where it follows from the Nielsen--Schreier Theorem~\cite[6.1.1]{Rob}.

A group $\Gamma$ is called \emph{residually finite} if the intersection of all
normal subgroups of finite index in $\Gamma$ is trivial.
Equivalently, $\Gamma$ embeds in a product of finite groups.
Every finitely generated abelian group is residually finite, and subgroups of residually finite
groups are again residually finite.

A group $G$ is \emph{virtually 'X'} (where 'X' is a group-theoretic property) if 
$G$ has a finite index subgroup which has property 'X'. 
\footnote{Finite group theory is therefore concerned with virtually trivial groups---this
indicates the drawbacks of unrestricted virtualization.}

A group element $g$ is called \emph{divisible}
if the equation $x^n=g$ has a solution in the given group for every integer $n\geq 1$.
A group is called divisible if it consists of divisible elements, and 
\emph{divisibly generated} if it is generated by a set of divisible elements.
Divisibility will play an important role in some proofs.
The following group-theoretic facts illustrate some of the consequences
and non-consequences of divisibility.
\subsection{Some facts about divisibility in groups}\label{FirstFacts}\ 

{\em
\begin{enumerate}[(i)]
\item Homomorphisms preserve divisibility.
\item
A group of finite exponent (e.g. a finite group) contains no divisible elements besides the identity.
Hence a residually finite group contains no divisible elements besides the identity. In particular,
a finitely generated abelian group contains no divisible elements besides the identity.
\item 
There exist finitely generated divisible groups. Moreover, every countable
group (e.g.~the additive group $\mathbb Q$) is isomorphic to a subgroup of some $2$-generated group.
\item
A connected locally compact group is divisibly generated.
\end{enumerate}}
\begin{proof}
Claim (i) is obvious. For (ii) we note that if a group $H$ has finite exponent $m>1$ and if $h\in H$ is
different from the identity, then the  equation $x^m=h$ cannot be solved in~$H$. 
The claims (iii) and (iv) are much deeper.
For (iii), see \cite{Guba} and \cite[Thm.~IV]{HNN} or \cite[6.4.7]{Rob}. Claim (iv) 
holds for all connected Pro-Lie groups, see \cite[Lem.~9.52]{HMPro}.
Every locally compact groups is a Pro-Lie group by the
Approximation Theorem due to Gleason, Montgomery--Zippin, and Yamabe \cite[p.~175]{MoZi}.
\end{proof}
When we are concerned with continuity questions, we view every countable group $\Gamma$ as a discrete
locally compact group. This is justified by the following elementary fact.
\begin{Lem}\label{CountableIsDiscrete}
A locally compact countable group is discrete.
\end{Lem}
\begin{proof}
By the Baire Category Theorem, some and hence every singleton $\{g\}\subseteq G$ is open,
see \cite[XI.10]{Dugundji}.
\end{proof}

\section{The case where $G$ is locally compact abelian}
For locally compact abelian groups we have the following general result.
Note that a finitely generated abelian group is residually finite.
\begin{Thm}
\label{Abelian}
Suppose that $G$ is a locally compact abelian group, that $G/G^\circ$ is 
topologically finitely generated, and 
that $\Gamma$ is a countable residually finite group. Suppose that
\[
     \phi:G\longrightarrow\Gamma
\]
is an abstract surjective group homomorphism. Then $\phi$ is continuous and open.
\end{Thm}
\begin{proof}
We proceed in four steps. The steps 1) and 2) below follow loosely the arguments in 
\cite[Thm.~5.1]{NS}.

\smallskip\noindent\emph{0) We may assume that $G^\circ=\{1\}$ and thus that $G$ is topologically finitely generated.}

\noindent
Since $G^\circ$ is divisible and since a residually finite group
contains no nontrivial divisible elements (see~\ref{FirstFacts}), $\phi$ is constant
on $G^\circ$. Hence $\phi$ factors through the continuous open homomorphism $G\longrightarrow G/G^\circ$.

\smallskip\noindent\emph{1) The claim is true if $G$ is compact and if $\Gamma$ is finite.}

\noindent
Let $F\subseteq G$ be a finitely generated dense subgroup and
let $m$ denote the exponent of $\Gamma$. 
Then $G^m=\{g^m\mid g\in G\}\subseteq G$ is a compact subgroup and the compact abelian group $G/G^m$ has finite exponent.
The image of $F$ in $G/G^m$ is finitely generated abelian of finite exponent and
therefore finite. On the other hand, $F$ has dense image in $G/G^m$. Therefore
$G/G^m$ is finite and thus $G^m$ is open in $G$. Now $G^m$ is contained in the
kernel of $\phi$, hence $\mathrm{ker}(\phi)$ is both open and closed. Therefore $\phi$ is continuous
and open.

\smallskip\noindent\emph{2) The claim is true if $G$ is compact.}

\noindent
Let $\mathcal N$ denote the collection of all normal subgroups of finite index in $\Gamma$.
For $\Delta\in\mathcal N$ let $\phi_\Delta$ denote the composite
$G\longrightarrow\Gamma\longrightarrow \Gamma/\Delta$.
By Step 1), $\mathrm{ker}(\phi_\Delta)$ is closed.
Hence $\mathrm{ker}(\phi)=\bigcap\{\mathrm{ker}(\phi_\Delta)\mid\Delta\in\mathcal N\}$
is closed.
Therefore $G/\mathrm{ker}(\phi)$ is a countable locally compact group. 
By Lemma~\ref{CountableIsDiscrete}, $G/\mathrm{ker}(\phi)$ is discrete and therefore
$\mathrm{ker}(\phi)$ is open and $\phi$ is continuous and open.

\smallskip\noindent\emph{3) The claim is true in general.}

\noindent
We may decompose the locally compact abelian group $G$ 
topologically as $G=\RR^n\times H$, where $H$ has a compact open subgroup $K$,
see \cite[Thm.~7.57~(i)]{HMCompact}. 
However, $G^\circ=\{1\}$ by Step 0), whence $G=H$ has a compact open subgroup $K$.
Let $\Gamma_0=\phi(K)$. Then $\Gamma_0\subseteq\Gamma$ is residually finite and countable.
The restriction $\phi:K\longrightarrow\Gamma_0$ is continuous and open by 
Step 2). Since $K\subseteq H$ is open, 
$\phi:H\longrightarrow\Gamma$ is continuous and open.
\end{proof}

The following examples show that the assumptions on $G$ and $\Gamma$ cannot be dropped.
\begin{Ex}
As an abstract group, the $1$-torus $\mathrm U(1)\cong\mathbb R/\mathbb Z$ is isomorphic to a product of Pr\"ufer groups and
a rational vector space of uncountable dimension, see e.g.~\cite[A1.43]{HMCompact}. In particular, there exist (many) 
non-continuous
surjective homomorphisms $\mathrm U(1)\longrightarrow\mathbb Q$. Hence the condition that $\Gamma$ is 
residually finite cannot be dropped. 

The compact group $G=\prod_{i=0}^\infty\mathbb F_2$ admits many non-continuous surjective homomorphisms
onto $\mathbb F_2$ (where $\mathbb F_2$ denotes the field of two elements).
Hence the condition that $G/G^\circ$ is topologically finitely generated can also not be dropped.
\end{Ex}
On the other hand, it is easy to extend the previous result to Pro-Lie groups.
\begin{Cor}
\label{ProLieAbelian}
Suppose that $G$ is an abelian Pro-Lie group, that $G/G^\circ$ is 
locally compact and topologically finitely generated, and 
that $\Gamma$ is a countable residually finite group. Suppose that
\[
     \phi:G\longrightarrow\Gamma
\]
is an abstract surjective group homomorphism. Then $\phi$ is continuous and open.
\end{Cor}
\begin{proof}
The identity component $G^\circ$ is a connected Pro-Lie group \cite[3.29]{HMPro},
and therefore divisible \cite[Lem.~9.52]{HMPro}. Hence $\phi$ factors through the open
homomorphism $G\longrightarrow G/G^\circ$. The claim follows now from
Theorem \ref{Abelian}.
\end{proof}

\section{The case where $G$ is an almost connected Lie group}
Now we consider an abstract surjective group homomorphism
\[
 G\longrightarrow\Gamma,
\]
where $G$ is an almost connected\footnote{i.e.~$G/G^\circ$ is finite}
Lie group and $\Gamma$ is a finitely
generated group. 
The next result 
was observed (in different degrees of generality) by Goto \cite{Goto},
Ragozin \cite{Ragozin} and George Michael \cite{Michael}. See also
\cite[Thm.~9.90]{HMCompact} and \cite[94.21]{CPP}.
For the sake of completeness, we include a proof.
\begin{Thm}\label{NormalInSemisimple}
Let $G$ be a Lie group. If the Lie algebra $\mathrm{Lie}(G)$ is semisimple, then
every abstract normal subgroup $N\trianglelefteq G$ is closed.
\end{Thm}
\begin{proof}
Let $N\trianglelefteq G$ be a normal subgroup.
Since $G^\circ\subseteq G$ is closed and open, it suffices to show that $N\cap G^\circ$ is closed
in $G^\circ$ in order to show that $N$ is closed. 
Hence we may assume that $G$ is connected, and we proceed by induction on the dimension of 
the compact connected semisimple Lie group $G$. 
The claim holds for trivial reasons if $\dim(G)=0$. 
For $\dim(G)>0$ we view $N\subseteq G$ as a topological group with respect to the subspace topology.
Let $N_1\subseteq N$ denote the path component of the identity.

If $N_1=\{1\}$, then the path component of every element $n\in N$ is trivial.
Then $\{gng^{-1}\mid g\in G\}=\{n\}$ holds for every
$n\in N$. Thus $N$ is contained in the closed discrete group
$\mathrm{Cen}(G)\subseteq G$ and therefore $N$ is closed.

If $N_1$ is nontrivial then there exists, by Yamabe's Theorem
(see~\cite{Goto} or \cite[Thm.~9.6.1]{HilgertNeeb}),
a Lie group structure on $N_1$ such that
$N_1\hookrightarrow G$ is a continuous injection.
In other words, $N_1\hookrightarrow G$ is a nontrivial connected analytic/virtual/integral Lie subgroup of $G$.
Its Lie algebra $\mathrm{Lie}(N_1)$ is then a nontrivial ideal in 
$\mathrm{Lie}(G)$. Since $\mathrm{Lie}(G)$ is
semisimple, the virtual subgroup corresponding to any ideal in $\mathrm{Lie}(G)$ is closed: it is
the connected centralizer of the complementary ideal.
Therefore $N_1\subseteq G$ is closed and thus a closed Lie subgroup of $G$. 
Now we may apply the induction hypothesis to 
$N/N_1\subseteq G/N_1$. The group $N$ is the preimage of $N/N_1$ under the continuous homomorphism
$G\longrightarrow G/N_1$ and therefore closed in $G$.
\end{proof}
\begin{Cor}
If $G$ is a connected Lie group whose Lie algebra is simple, then every proper abstract normal subgroup
of $G$ is contained in the center of $G$. In particular, $G/\mathrm{Cen}(G)$ is simple as an
abstract group.
\qed
\end{Cor}
\begin{Cor}\label{NoCountabeIndexInSemisimple}
A connected semisimple Lie group has no proper abstract normal subgroups of countable index.
\qed
\end{Cor}

Now we come to the main result of this section. A variation is proved in Proposition~\ref{LieGroupCountableIndexNormal}.
\begin{Thm}\label{QuotientsOfLieGroups}
Let $G$ be a Lie group, with $G/G^\circ$ finite. Let $\Gamma$ be a finitely generated group
and let 
\[
\phi:G\longrightarrow\Gamma
\]
be an abstract surjective homomorphism. Then $\Gamma$ is finite and $\phi$ is continuous and open.
\end{Thm}
\begin{proof}
The subgroup $\phi(G^\circ)\subseteq\Gamma$ has finite index and is
therefore finitely generated.
Hence it suffices to prove that $\phi$ is constant if $G$ is connected.

Suppose that $G$ is connected.
Let $R\trianglelefteq G$ denote the solvable radical, i.e. the unique maximal connected solvable closed normal
subgroup of $G$, see~\cite[Prop.~16.2.2.]{HilgertNeeb} and \cite[Def.~94.18, Prop.~94.19]{CPP}. 
Then the Lie algebra $\mathrm{Lie}(G/R)$ is semisimple. We put $N=\mathrm{ker}(\phi)$.
By Theorem~\ref{NormalInSemisimple},  $NR\subseteq G$ is closed and therefore a Lie subgroup. 
Since $\Gamma\cong G/N$ is countable, $G/NR$ is a countable
and hence discrete Lie group. It follows that $NR\subseteq G$ is open,
hence $G=NR$. Now we have $\Gamma\cong NR/N\cong R/R\cap N$.
Thus $\Gamma$ is solvable.
The connected Lie group $R$ is divisibly generated (see~\ref{FirstFacts}).
It follows that the abelianization $\Gamma_{ab}$ of $\Gamma$ is both divisible and
finitely generated abelian, whence $\Gamma_{ab}=\{1\}$ (again by~\ref{FirstFacts}). 
On the other hand, $\Gamma$ is solvable,
hence $\Gamma_{ab}\neq\{1\}$ if $\Gamma\neq\{1\}$. Hence $\Gamma=\{1\}$.
\end{proof}
\begin{Prob}
Is the
conclusion of Theorem~\ref{QuotientsOfLieGroups} still true if we drop the assumption that $G/G^\circ$ is finite?
\end{Prob}

\section{Quotients of locally compact groups}
We now extend the results from the previous sections to locally compact groups.
Here the fundamental work of Nikolov--Segal \cite{NS}, as well as the solution
of the 5th Hilbert problem, enter in an essential way. But first we draw a simple 
consequence from the facts collected in~\ref{FirstFacts}.
\begin{Lem}\label{FiniteIndexLieLemma}
Let $G$ be a connected locally compact group and let $E$ be a group of finite
exponent (e.g. a finite group), or a residually finite group. 
Then every abstract homomorphism
$G\longrightarrow E$ is trivial. In particular, a connected locally compact group
has no abstract proper subgroups of finite index, and no nontrivial residually finite quotients.
\end{Lem}
\begin{proof}
By \ref{FirstFacts}, the group $G$ is divisibly generated, and
$E$ contains no divisible elements. 
Hence every homomorphism $G\longrightarrow E$ is constant.
For the last assertion we note that
a group $G$ which has a proper finite index 
subgroup $H$ maps nontrivially into the finite symmetric group
$\mathrm{Sym}(G/H)$. Hence the existence of proper finite index subgroups implies the
existence of proper finite index normal subgroups.
\end{proof}
The situation is quite different for non-normal subgroups of countable index: in 
connected Lie groups, such groups always exist.
\begin{Thm}
\label{Kallman}
The group $\mathrm{SL}_n\mathbb C$ has subgroups of infinite countable index.
Therefore every connected Lie group $G$ has subgroups of infinite countable index.
\end{Thm}
\begin{proof}
This is proved in \cite[Thm.~2.1]{ST2} and in \cite[Thm.~1]{Kallman}. The algebraic-geometric reason 
is the following.
Let $p$ be a fixed prime.
The group $\mathrm{SL}_n\mathbb Q_p$ acts on its Bruhat--Tits building $X$,
which has a countable set of vertices. Thus $\mathrm{SL}_n\mathbb Q_p$
has subgroups of countable index, namely the parahoric subgroups.

Now let $K$ denote the algebraic closure of $\mathbb Q_p$.
The discrete valuation on $\mathbb Q_p$ extends to a nondiscrete
valuation on $K$, and there is a corresponding nondiscrete euclidean
building $X_K$ on which $\mathrm{SL}_nK$ acts. This building is not simplicial
any more, but nevertheless it has a countable invariant subset of 'vertices'.
Thus $\mathrm{SL}_nK$ has subgroups of countable index.

Next we note that all algebraically closed fields of characteristic $0$ and
cardinality $2^{\aleph_0}$ are isomorphic to $\mathbb C$. Thus 
$\mathrm{SL}_n\mathbb C\cong\mathrm{SL}_nK$ has abstract subgroups of countable index.

Finally, every connected Lie group $G$ admits a nontrivial representation
$G\longrightarrow \mathrm{SL}_n\mathbb C$, for some $n$. Thus $G$ has nontrivial
abstract subgroups of countable index as well, and the index cannot be finite by
Lemma~\ref{FiniteIndexLieLemma}.
\end{proof}
More generally, one may look for subgroups in a locally compact group
$G$ which are not Borel sets, or even not measurable with respect to the Haar measure.
Note that the $\sigma$-additivity of the Haar measure implies that subgroups of countable
index cannot be measurable, so the subgroups appearing in Theorem~\ref{Kallman} cannot
be measurable.
The existence of non-measurable subgroups in general compact
groups is studied in detail in a recent article by 
Hern\'andez--Hofmann--Morris \cite{HHM}.

\smallskip
In contrast to Theorem~\ref{Kallman}, we have the following result for \emph{normal} subgroups
of countable index. For the next results in this section, 'countable' may be replaced
by 'of cardinality strictly less than $2^{\aleph_0}$' without changing any of the arguments.
\begin{Prop}
\label{LieGroupCountableIndexNormal}
Let $G$ be a Lie group, with $G/G^\circ$ finite. Suppose that
$N\trianglelefteq G$ is an abstract normal subgroup of countable index.
Then $G/N$ is virtually solvable. If $G$ is connected, then $G/N$ is solvable.
\end{Prop}
\begin{proof}
Since the image of $G^\circ$ in $G/N$ has finite index,
it suffices to show that the image of $G^\circ$ in $G/N$ is solvable.
So we may assume that $G$ is connected. As in the proof of Theorem~\ref{QuotientsOfLieGroups},
let $R\trianglelefteq G$ denote the solvable radical of $G$, Then $NR/R$ is a normal
subgroup of countable index in the semisimple Lie group $G/R$. 
By Corollary~\ref{NoCountabeIndexInSemisimple} we have $NR=G$. Thus $G/N=NR/R\cong R/N\cap R$ is solvable.
\end{proof}
In order to extend this result to general locally compact groups, a crucial
ingredient is Iwasawa's splitting theorem. The global version of this result is an extremely
useful answer to Hilbert's 5th problem.
\begin{Thm}[The Global Splitting Theorem]
\label{GST} 
Let $G$ be a locally compact group and let $U\subseteq G$ be a neighborhood of the identity. Then
there is a compact  subgroup $K\subseteq G$ contained in $U$, a simply connected 
Lie group $L$, and an open and continuous homomorphism 
\[\psi:L\times K\longrightarrow G\] with discrete kernel such that $\psi(1,k)=k$
for all $k\in K$.
\end{Thm}
\begin{proof}
The local version of this result, assuming $G$ to be connected and approximable by Lie groups,
is due to Iwasawa~\cite{Iwasawa}. See also Glu\v skov~\cite{Gluskov}. The present version (and some more
remarks on the history) can be found in Hofmann--Morris~\cite[Thm.~4.1]{hofmori} 
and in \cite[Thm.~4.4]{HofmannKramer}.
\end{proof}
The next result by Nikolov--Segal 
is considerably more difficult to prove than Proposition~\ref{LieGroupCountableIndexNormal}. 
A proof will be given in Section 6 below.
\begin{Thm}[Nikolov--Segal]\label{N-S1}
Let $G$ be a compact group, with $G/G^\circ$ finite. Suppose that
$N\trianglelefteq G$ is an abstract normal subgroup of countable index.
Then $G/N$ is virtually abelian. If $G$ is connected, then $G/N$ is
abelian. If $G/N$ is residually finite, then $N$ is closed.
\end{Thm}
Our final result about quotients of locally compact groups is as follows.
\begin{Thm}
Let $G$ be a locally compact group, with $G/G^\circ$ compact.
Suppose that $\Gamma$ is a finitely generated group and that 
\[
 \phi:G\longrightarrow\Gamma
\]
is an abstract surjective homomorphism. Then $\Gamma$ is finite. If $G/G^\circ$ is topologically
finitely generated, then $\phi$ is continuous and open.
\end{Thm}
\begin{proof}
Let $\Gamma_0=\phi(G^\circ)$ and $N=\mathrm{ker}(\phi)$. 
Since $G/G^\circ$ is compact and since $G/NG^\circ$ is finitely generated, the main result 
\cite[Thm.~5.25]{NS} implies
that $G/NG^\circ\cong\Gamma/\Gamma_0$ is finite. Thus $\Gamma_0$ is also finitely generated.
We put $\psi:L\times K\longrightarrow G^\circ$ as in Theorem~\ref{GST} and we note that this homomorphism is
necessarily surjective. We consider the composite
\[
\begin{tikzcd}
L\times K \arrow{rd}{\bar\phi} \arrow{r}{\psi} & G^\circ \arrow{d}{\phi} \\
& \Gamma_0.
\end{tikzcd}
\]
The connected Lie group $L\times\{1\}$ maps onto the finitely
generated group $\Gamma_0/\bar\phi(\{1\}\times K)$. By Theorem~\ref{QuotientsOfLieGroups}, we have
$\bar\phi(\{1\}\times K)=\Gamma_0$. In particular, $K$ maps onto the finitely generated
group $\Gamma_0$. Again by \cite[Thm.~5.25]{NS}, the quotient $\Gamma_0$ is finite. Since $G^\circ$ is divisibly
generated (see \ref{FirstFacts}), the group $\Gamma_0$ is trivial and $\Gamma$ is finite.
Moreover, $\phi$ factors through the projection $G\longrightarrow G/G^\circ$. 
Again by \cite[Thm.~5.25]{NS}, the composite 
$G\longrightarrow G/G^\circ\longrightarrow\Gamma$ is continuous if $G/G^\circ$ is topologically finitely generated.
\end{proof}
The previous proof relies heavily on Nikolov--Segal \cite[Thm.~5.25]{NS}, which is stronger
than Theorem~\ref{N-S1} above. This result depends, among other things, on the classification
of the finite simple groups, and on their advanced structure theory.

\section{Products and ultraproducts of nonabelian simple groups}

In order to prove Theorem~\ref{N-S1}, we will have to consider normal
subgroups in infinite products of compact semisimple Lie groups.
In this section we take a look at normal subgroups in products of abstract groups.
We will see that this leads almost inevitably to questions about ultraproducts.

If $G_1,\ldots,G_m$ are nonabelian simple groups, then it is easy to see that
the product $G_1\times\cdots\times G_m$ is a perfect group, and that every 
proper normal subgroup of the product is contained in the kernel of at least one projection map 
\[pr_j:G_1\times\cdots\times G_m\longrightarrow G_j.\]
But both questions become considerably more interesting for infinite products of groups.


First we look at products of perfect groups.
Given a perfect group $H$, let $c(H)$ denote the \emph{commutator length} of $H$,
that is, the smallest integer $n$ such that
every element $h\in H$ can be written as a product of at most $n$ commutators.
If no such integer exists, we put $c(H)=\infty$.
\begin{Lem}\label{PerfectLemma}
Let $(G_i)_{i\in I}$ be a family of perfect groups. Put $I_\infty=\{i\in I\mid c(G_i)=\infty\}$.
Then the product $\prod_{i\in I}G_i$ is perfect if and only if $I_\infty$ is finite and
if the function $i\mapsto c(G_i)$ is bounded on $I_0=I\setminus I_\infty$ .
\end{Lem}
\begin{proof}
Suppose that $I_\infty$ is finite. A finite product of perfect groups is perfect,
hence $L=\prod_{j\in I_\infty}G_j$ is perfect. If the function
$k\mapsto c(G_k)$ is bounded on $I_0$ by a constant
$n\in\mathbb N$, then every element in $M=\prod_{k\in I_0}G_k$ is a product
of at most $n$ commutators, and hence $M$ is perfect as well. Therefore 
$\prod_{i\in I}G_i=L\times M$ is perfect.

If $I_\infty$ is infinite or if the function $k\mapsto c(G_k)$ is unbounded on $I_0$, then
we can find an injection $\iota:\mathbb N\longrightarrow I$ such that
for every $n\in\mathbb N$ there is an element $g_{\iota(n)}\in G_{\iota(n)}$
which cannot be written as a product of less than $n$ commutators.
Put $g_j=1$ if $j\not\in \iota(\mathbb N)$. Then $(g_i)_{i\in I}$ is
an element in $\prod_{i\in I}G_i$ which is not a product of commutators.
Therefore $\prod_{i\in I}G_i$ is not perfect.
\end{proof}
In view of the previous lemma, we should expect to encounter lots of 'strange'
normal subgroups in infinite products of nonabelian simple groups.
In particular we cannot expect to catch all proper normal subgroups by
the projection maps
\[
 pr_j:\prod_{i\in I} G_i \longrightarrow G_j.
\]
But if we replace the finiteness of the index set $I$ by a compactness condition, we gain
some control over normal subgroups in the product. This leads directly to ultraproducts, 
as we will see.
We first change our viewpoint slightly.
\subsection{Presheaves of groups}
Recall that a \emph{presheaf of groups} $\mathcal G$ on a topological space $X$
is a contravariant group functor \[\mathcal G:\mathbf{Open}(X)\longrightarrow \mathbf{Group}\]
which assigns to every open set $U\subseteq X$ a group $\mathcal G(U)$.
For  open sets
$U\subseteq V\subseteq X$, there are restriction homomorphisms $pr^V_U:\mathcal G(V)\longrightarrow\mathcal G(U)$,
subject to the usual compatibility condition $pr_U^V\circ pr_V^W=pr_U^W$, for $U\subseteq V\subseteq W$,
and with $pr_U^U=\id_{\mathcal G(U)}$.
Recall also that the \emph{stalk} at a point $p\in X$ is defined as the direct limit
\[
 \mathcal G(p)=\lim_\rightarrow\{\mathcal G(U)\mid U\in\mathbf{Open}(U)\text{ and }p\in U\}.
\]
For every $p\in U\in\mathbf{Open}(X)$, there is a natural restriction homomorphism
\[pr_p^U:\mathcal G(U)\longrightarrow\mathcal G(p).\]
Given a continuous map $\phi:X\longrightarrow Y$ between topological spaces, 
there is the \emph{direct image presheaf} $\phi_*\mathcal G$ on $Y$
which assigns to $W\in\mathbf{Open}(Y)$ the group \[\phi_*\mathcal G(W)=\mathcal G(\phi^{-1}(W)).\]
We note that $\phi_*\mathcal G(Y)=\mathcal G(X)$.

\subsection{Products of groups as presheaves}
We may view the family of groups $(G_i)_{i\in I}$ as the group-valued presheaf
$\mathcal G$ on the discrete topological space $I$ which assigns to every
nonempty  subset $J\subseteq I$ the group $\mathcal G_J=\prod_{j\in J}G_j$,
and to $\emptyset\subseteq I$ the trivial group.
The stalks are then the groups $\mathcal G(i)=G_i$. 
The problem that not all normal subgroups in $\prod_{i\in I}G_i$ can be detected
by the projection maps $pr_j:\prod_{i\in I}G_i\longrightarrow G_j$ can be rephrased as
follows: 
\begin{quote}
 \emph{There are 
not enough stalks in the space $I$ in order to detect all normal subgroups in the product $\prod_{i\in I}G_i$. }
\end{quote}
We remove this deficiency by 'enlarging' $I$ through a map $I\longrightarrow X$,
where $X$ is a compact space.
\begin{Prop}\label{StalksDetectNormal}
Suppose that $(G_i)_{i\in I}$ is a family of groups and that $N\trianglelefteq \prod_{i\in I}G_i$
is a normal subgroup.
Let $X$ be a compact space and let $\iota:I\longrightarrow X$ be any map. 
If $N$ maps under the composite homomorphism
\[
N\hookrightarrow\prod_{i\in I}G_i\xrightarrow{\ \cong\ }\iota_*\mathcal G(X)\longrightarrow\iota_*\mathcal G(p)
\]
onto $\iota_*\mathcal G(p)$, for every 
$p\in X$, then $N$ contains the derived group of $\prod_{i\in I}G_i$.
\end{Prop}
\begin{proof}
We identify $\prod_{i\in I}G_i$ with $\iota_*\mathcal G(X)$.
Assuming that $N$ surjects onto all stalks $\iota_*\mathcal G(p)$, and
given any two elements $g,h\in\iota_*\mathcal G(X)$, we will show that $N$ contains $[g,h]$.

Let $p\in X$. Since $N$ maps under $pr_p^X$ onto $\iota_*\mathcal G(p)$,
there is an open neighborhood $U_p\subseteq X$ of $p$ and an element
$n_p\in N$ such that $pr^X_{U_p}(n_p)=pr^X_{U_p}(g)$.
Since $X$ is compact, we can cover $X$ by finitely many such sets
$U_{p_1},\ldots,U_{p_r}$. For $s=1,\ldots,r$, put $n_s=n_{p_s}$, and put
\[I_s=\iota^{-1}(U_{p_s})\quad\text{and}\quad
I'_s=I_s\setminus(I_1\cup\cdots\cup I_{s-1}).
\]
Then 
\[
I=I'_1\dot\cup\cdots\dot\cup I'_r\quad\text{and}\quad \iota(I_s)\subseteq U_{p_s}.
\]
For $s=1,\ldots,r$ we define elements $m_s\in\prod_{i\in I}G_i$ by
\[(m_s)_i=\begin{cases}
       h_i&\text{ if }i\in I'_s\\ 1&\text{ else.}
      \end{cases}
\]
Then
\[
 ([n_s,m_s])_i=\begin{cases}
                ([g,h])_i&\text{ if }i\in I_s'\\
                1&\text{ else.}
               \end{cases}
\]
Therefore $[g,h]=[n_1,m_1]\cdots[n_r,m_r]\in N$.
\end{proof}
In a completely analogous way, we have the following result
\begin{Prop}\label{StalksDetectPerfect}
Let $X$ be a compact space, let $(G_i)_{i\in I}$ be a family of groups, and
let $\iota:I\longrightarrow X$ be a map. Then $\prod_{i\in I}G_i$ is perfect if and only
if every stalk $\iota_*\mathcal G(p)$ is perfect.
\end{Prop}
\begin{proof}
First of all we note that it is clear from the construction of $\iota_*\mathcal G$ 
that the restriction homomorphisms $pr^U_V$ are always surjective
(such a presheaf is called \emph{flabby}), hence $\mathcal G(X)$ maps onto $\mathcal G(p)$, for every
$p\in X$. The image of a perfect group is perfect, so one implication
is clear.

Now let $g\in\prod_{i\in I}G_i$ and assume that every stalk $\iota_*\mathcal G(p)$ is perfect.
Hence for every $p\in X$ there are elements $a_1,b_1,\ldots,a_n,b_n\in\prod_{i\in I}G_i$
such that $pr_p^X([a_1,b_1]\cdots[a_n,b_n])=pr_P^X(g)$. Again, this relation holds then in some open
neighborhood $U_p$ of $p$. Now $X$ is covered by finitely many such sets $U_1=U_{p_1},\ldots, U_m=U_{p_m}$,
where \[pr_{U_s}^X(g)=pr_{U_s}^X[a_{s,1},b_{s,1}]\cdots[a_{s,n},b_{s,n}].\]
The number $n$ can be chosen uniformly.
By a similar argument as in the proof of Proposition~\ref{StalksDetectNormal}, there is a disjoint
decomposition $I=I_1'\cup\cdots\cup I_m'$ with $\iota(I'_s)\subseteq U_s$.
We put
\[ 
\widetilde{(a_{s,t})}_i=\begin{cases} (a_{s,t})_i&\text{ if }i\in I'_s\\ 1&\text{ else} \end{cases}
\quad \text{ and }\quad 
\widetilde{(b_{s,t})}_i=\begin{cases} (b_{s,t})_i&\text{ if }i\in I'_s\\ 1&\text{ else} \end{cases}.
\]
Then $g=[\widetilde{a_{1,1}},\widetilde{b_{1,1}}]\cdots[\widetilde{a_{m,n}},\widetilde{b_{m,n}}]$.
\end{proof}
Now we look at specific choices for $X$.
There are (a least) three obvious candidates.
\begin{Ex}[The trivial case]
If $X=\{p\}$ is a singleton and $\iota$ is the constant map, then $\iota_*\mathcal G(p)=\prod_{i\in I}G_i$.
Obviously, the Propositions~\ref{StalksDetectNormal} and \ref{StalksDetectPerfect} tell us nothing new.
\end{Ex}
\begin{Ex}[The cofinite case]
If $X=\hat I =I\cup\{\infty\}$ is the Alexandrov compactification of $I$ and $\iota:I\longrightarrow \hat I$ is the inclusion,
then $\mathrm{ker}(pr^{\hat I}_\infty)$ is the restricted product
\[
 \prod_{i\in I}^{res}G_i=\{(g_i)_{i\in I}\mid g_i=1\text{ for almost all }i\}\trianglelefteq\prod_{i\in I}G_i.
\]
Thus $\prod_{i\in I}G_i$ is perfect if and only if each $G_i$ is perfect, and if in addition
$\prod_{i\in I}G_i\big/\prod_{i\in I}^{res}G_i$ is perfect (compare this with Lemma \ref{PerfectLemma}).

If $\prod_{i\in I}G_i$ is perfect and if $N\trianglelefteq \prod_{i\in I}G_i$ is a proper normal subgroup
with $pr_i^X(N)=G_i$ for all $i\in I$ (an example of such a group is $N=\prod_{i\in I}^{res}G_i$),
then $N\big(\prod_{i\in I}^{res}G_i\big)\neq\prod_{i\in I}G_i$.
\end{Ex}
\begin{Ex}[The ultraproduct case]
The real case of interest is however when $X=\beta I$ is the \v Cech--Stone compactification
of $I$, with the standard inclusion $\iota:I\longrightarrow\beta I$.
Recall that $\beta I$ can be identified with the set of all ultrafilters on $I$.
The points $j\in I$ correspond to the principal ultrafilters via
$j\mapsto\{J\subseteq I\mid j\in J\}$, whereas the 'new' points
in $\beta I$ correspond to the free ultrafilters\footnote{also called non-principal ultrafilters} on $I$. 
For $J\subseteq I$ put
$J^*=\{\mu\in\beta I\mid J\in\mu\}$. Then $\{J^*\mid J\subseteq I\}$ is a basis
for the topology on $\beta I$.
Given $\mu\in\beta I$, the set $\{J^*\mid J\in\mu\}$ is a neighborhood basis of $\mu$.
It follows that $\iota_*\mathcal G(\mu)=G_j$ if $\mu$ is the principal
ultrafilter generated by $j$. But if $\mu$ is a free ultrafilter, then 
$\iota_*\mathcal G(\mu)$
is the \emph{ultraproduct} (see~\cite[Ch.~5]{BS})
\[
\iota_*\mathcal G(\mu)=\prod_{i\in I} G_i\big/\mu,
\]
where we identify two sequences $(g_i)_{i\in I}$, $(h_i)_{i\in I}$ if
$\{j\mid\ g_j=h_j\}\in\mu$. 
\footnote{One may think of the elements of $\mu$ as sets of measure $1$. Thus two
sequences are considered to be equivalent if they differ only on a set of measure $0$.}
We write 
\[pr_\mu=pr_\mu^{\beta I}:\prod_{i\in I}G_i\longrightarrow\prod_{i\in I}G_i\big/\mu\] 
for the corresponding projection map.

By Proposition~\ref{StalksDetectPerfect}, the product
$\prod_{i\in I}G_i$ is perfect if and only if each $G_i$ is perfect and if
for every free ultrafilter $\mu$ on $I$, the ultraproduct $\prod_{i\in I} G_i\big/\mu$ is perfect.

If the product $\prod_{i\in I}G_i$ is perfect and if $N\trianglelefteq \prod_{i\in I}G_i$ is a proper
normal subgroup, then by Proposition~\ref{StalksDetectNormal} either there exists 
an index $j\in I$ with $pr_j(N)\neq G_j$, or there exists a free ultrafilter $\mu$
such that the image of $N$ in $\prod_{i\in I} G_i\big/\mu$ is different from
$\prod_{i\in I}G_i\big/\mu$.
\end{Ex}
We record this last result for later use.
\begin{Prop}\label{UltraProductProp}
Let $(G_i)_{i\in I}$ be a family of groups. Suppose that the product
$\prod_{i\in I}G_i$ is perfect and that $N\trianglelefteq \prod_{i\in I} G_i$ is a proper normal
subgroup. Then either there exists an index $j\in I$ with $pr_j(N)\neq G_j$
or there exists a free ultrafilter $\mu$ on $I$ such that $pr_\mu(N)\neq\prod_{i\in I} G_i\big/\mu$.
\qed
\end{Prop}
Products and ultraproducts of finite simple groups and their normal subgroups have been
extensively studied. See e.g. \cite{ST1}, \cite{ST2}, \cite{Point} and the literature
quoted there. The normal subgroup
lattice of ultraproducts of compact simple Lie groups is studied in detail in \cite{StolzThom}.

\section{Normal generation of compact simple Lie groups}

We call a Lie group $G$ \emph{almost simple} if it is connected
and if its Lie algebra is simple.
Let $G$ be a compact almost simple Lie group.
Then the center $\mathrm{Cen}(G)$ is finite
and $G/\mathrm{Cen}(G)$ is simple as an abstract group by Theorem~\ref{NormalInSemisimple}.
Moreover, the conjugacy class 
\[C(a)=\{gag^{-1}\mid g\in G\}\]
of every non-central element $a\in G$ generates $G$.
We will now re-prove this in a completely different way. The new approach 
has the advantage that it says, in a quantitative way, how fast a conjugacy
class generates the whole group.
\subsection{Some observations}
Suppose that $G=\mathrm{SU}(n)$. The conjugacy class of a matrix $a\in \mathrm{SU}(n)$ is
determined by its spectrum (with its multiplicities). Given such a
spectral datum, the task is therefore to write an arbitrary element
$g\in \mathrm{SU}(n)$ as a product $g=a_1\cdots a_m$ of special unitary matrices 
$a_1,\ldots,a_m$ with a prescribed spectrum. 
\footnote{This problem is ultimately related to Horn
inequalities, Schubert calculus and buildings, see e.g.~\cite{Woodward}, \cite{JM} and \cite{KLM}.}

This viewpoint shows also that there are lower bounds on the number of matrices 
needed. Consider for example the group $G=\mathrm{SO}(2k+1)$, for $k\geq 1$. Every element in $g\in\mathrm{SO}(2k+1)$
can be written as a product of an even number of reflections.
Let $g\in\mathrm{SO}(2k+1)$ be a matrix with exactly one real eigenspace $E\subseteq\mathbb R^{2k+1}$, 
of dimension $\dim(E)=1$.
Then $g$ cannot be written as a product of less than $2k$ be reflections. 
To see this, note that every reflection fixes a hyperplane.
The intersection of less than $2k$ hyperplanes in $\mathbb R^{2k+1}$
has dimension at least~$2$, hence a product 
of less than $2k$ reflections fixes a $2$-dimensional subspace pointwise. The reflections
themselves are not in $\mathrm{SO}(2k+1)$, but their negatives are.
Hence we have here a lower bound on the number of multiplications of a fixed conjugacy class
needed, which grows with the rank of the group.

\smallskip
We will prove in Theorem~\ref{KThm}
that there exists a rational real function $f$, and 
on every simply connected almost simple compact Lie group $G$ a nonnegative bounded continuous real \
class function $\hat\sigma$
with the following property. If $\hat\sigma(a)>\theta$, then every element in
$G$ is a product of less than $f(\theta)$ conjugates of $a,a^{-1}$.
This is a first-order property of compact Lie groups, which extends by \L os' Theorem 
to ultraproducts. Combining Proposition~\ref{UltraProductProp} with the Structure Theorem for compact 
connected groups, we obtain then a proof of Theorem~\ref{N-S1}.

\smallskip
In what follows, we put for a nonnegative real number $r$
\[
\lceil r\rceil=\min\{k\in\mathbb N\mid k\geq r\}
\]
and we note that
\[
 \lceil r\rceil<r+1.
\]
Our whole approach follows closely the reasoning in Nikolov--Segal \cite{NS}, pp.~582--595,
combined with a tweak introduced in \cite{StolzThom}.
However we avoid the notion of \emph{asymptotic cones of metric spaces} (the \emph{ultralimits} in \emph{loc.cit.})
altogether. Instead we work bare hand with ultraproducts.

\subsection{Some elementary calculations.}\label{Elementary}
We first consider the case of the almost simple matrix group
\[
 G=\mathrm{SU}(2)=\left.\left\{ 
 \left(\begin{smallmatrix}  c & s\\-\bar s & \bar c \end{smallmatrix}\right)\in\mathbb C^{2\times2}\right| c\bar c+s\bar s=1\right\}.
\]
For $c=\pm1$ we obtain the $2\times2$ identity matrix $\bf1$ and its negative $-\bf1$, respectively.
Note that $\mathrm{Cen}(\mathrm{SU}(2))=\{\pm{\bf1}\}$.
We identify $\mathrm{SU}(2)$ with the round $3$-sphere in $\mathbb C^2\cong\mathbb R^4$, and we
endow it with the unique bi-invariant Riemannian metric $d$ of diameter $\pi$
(so $d$ is the spherical or angular metric).
We fix the maximal torus 
\[
  T=\left.\left\{\left(\begin{smallmatrix}  c & \\ & \bar c \end{smallmatrix}\right)\in\mathbb C^{2\times2}\right| c\bar c=1\right\}
  \subseteq\mathrm{SU}(2)
\]
and we note that every element in $\mathrm{SU}(2)$ is conjugate to an element in $T$.

Now we use some elementary spherical geometry.
The conjugation action of $G$ on itself is by rotations around the axis passing through $\pm\bf1$.
The conjugacy class $C(a)=\{gag^{-1}\mid g\in G\}$ of an element $a\in G\setminus\{\pm{\bf1}\}$ is thus a round $2$-sphere
of spherical radius $\theta$, where $\frac{\theta}{2}=\min\{d({\bf1},a),d(-{\bf1},a)\}\in[0,\frac{\pi}{2}]$.
Then $C(a)a^{-1}$ is a round sphere of the same shape, but passing through the identity
element ${\bf1}\in G$. The set of all conjugates of the set $C(a)a^{-1}$, which coincides obviously with $C(a)C(a^{-1})$,
is thus the closed metric ball \[C(a)C(a^{-1})=B_{\theta}({\bf1})\] consisting of all $b\in G$ whose spherical distance from 
$\bf1$ is at most $\theta$.
Next we note that we have for every $\theta\in[0,\pi]$ that 
\[G=\underbrace{B_{\theta}({\bf1})\cdots B_{\theta}({\bf1})}_{k\text{ factors}}\text{ whenever }k\theta\geq\pi.\]
\begin{quote}
\emph{In particular we have shown that every element in $\mathrm{SU}(2)$ is a product not more than
$2\lceil\frac{\pi}{\theta}\rceil$ conjugates of $a$ and $a^{-1}$.}
\end{quote}
The quantity $\theta$ can be defined in a more natural way. We consider the \emph{root} (or character) 
\[
\delta:T\longrightarrow\mathrm{U}(1),\quad 
\delta\left(\begin{smallmatrix}  c & \\ & \bar c \end{smallmatrix}\right)= c^2.
\] 
For $c\in\mathrm{U}(1)$ we denote by $\ell(c)\in[0,\pi]$ the spherical (or angular) distance 
between $1$ and $c$ on the unit circle. 
Thus $\ell(ab)\leq\ell(a)+\ell(b)$ and $\ell(\bar a)=\ell(a)$ hold for all $a,b\in\mathrm{U}(1)$.
Then we have for all $a\in T$ that the conjugacy class $C(a)$ is a sphere of diameter
\[
 \theta=\ell(\delta(a))
\]
(where a 'sphere of diameter $0$' is meant to be a point).
Using these elementary observations, we have the following 'bootstrap Lemma'.
\begin{Lem}\label{ElementaryLemma}
Let $H=H_0\cdot H_1\cdots H_s$ be a compact connected Lie group which is a central product of
groups $H_i\cong\mathrm{SU}(2)$, for $1\leq i\leq s$, and a torus $H_0$.
For each factor $H_i\cong\mathrm{SU}(2)$ 
with $1\leq i\leq s$ we fix a torus $T_i\subseteq H_i$ and a root $\delta_i:T_i\longrightarrow\mathrm{U}(1)$
exactly as in \ref{Elementary}. Suppose that $0<\theta\leq\pi$ and that
$h\in H$ is an element such that $\ell(\delta_i(h))\geq\theta$ holds for all $1\leq i\leq s$.
Then every element in the subgroup $H_1\cdot H_2\cdots H_s\subseteq H$ is a product of
at most $2\lceil \frac{\pi}{\theta}\rceil$ conjugates of $h$ and $h^{-1}$.
\end{Lem}
\begin{proof}
We write $h$ as a product $h=h_0h_1\cdots h_s$, with $h_i\in H_i$.
Then $C(h)C(h^{-1})$ is the product of the $s$ commuting sets $(C(h_i)C(h_i^{-1}))$,
with $1\leq i\leq s$. There appears no more factor in $H_0$, because $H_0$ is in the center
of $H$ and because $h_0$ and $h_0^{-1}$ cancel out.
The claim follows now from the observations in \ref{Elementary}, applied to each $H_i$ individually.
\end{proof}
Now we turn to root systems in
compact almost simple Lie  groups in general. The books \cite{Adams}, \cite{BtD}, \cite{HilgertNeeb} and \cite{HMCompact}
are excellent references for the facts that we need.
\subsection{Roots in a compact Lie group.}\label{Roots1}
Suppose that $G$ is a compact simply connected almost simple Lie group.
We fix a maximal torus $T\subseteq G$ and its character group $X^*(T)=\mathrm{Hom}(T,\mathrm{U}(1))$.
Associated to $T$ we have the root system $\Phi\subseteq X^*(T)$.
The special case of $G=\mathrm{SU}(2)$, where $\Phi=\{\pm\delta\}$, was described in \ref{Elementary}.

Let $\Delta\subseteq\Phi$ be a system of simple roots, $\Delta=\{\delta_1,\cdots,\delta_r\}$.
Then $r=\dim(T)$ is the \emph{rank} of $G$. There is a canonical isomorphism
$V=X(T)^*\otimes\RR\cong\mathrm{Hom}_{\mathbb R}(\mathrm{Lie}(T),\mathbb R)$, and 
$\Delta$ is a basis of this real vector space. There is also a canonical inner product
on $V$, and $\Phi$ is a reduced irreducible root system in the vector space $V$.
The reflection group generated by $\Phi$ is the Weyl group $W=W(\Phi)$.
The Weyl group is canonically isomorphic to $N/T$, where $N=\mathrm{Nor}_G(T)$ 
is the normalizer of $T$ in $G$.

There is also the cocharacter group $X_*(T)=\mathrm{Hom}(\mathrm{U}(1),T)$,
and a natural isomorphism $X_*(T)\otimes\mathbb R\cong\mathrm{Lie}(T)$.
Associated to every root $\alpha$ there is a coroot $\alpha^\vee$, with
$\alpha(\alpha^\vee(c))=c^2$. The action of $W(\Phi)$ on $\mathrm{Lie}(T)$ coincides
with the adjoint action of $N/T$.

\subsection{Some more definitions}
For a subset $\Delta'\subseteq\Delta$ we denote by $T_{\Delta'}$ the $\#\Delta'$-dimensional
subtorus
spanned by the coroots $\{\delta^\vee\mid\delta\in\Delta'\}$, i.e.
\[\textstyle
 T_{\Delta'}=\prod\{\delta^\vee(c)\mid\delta\in\Delta'\text{ and }c\in\mathrm{U}(1)\}\subseteq T.
\]
We now define real-valued continuous functions $\sigma$ and $\hat\sigma$ on $T$ by
\[\textstyle
 \sigma(h)=\frac{1}{r}\sum_{i=1}^r\ell(\delta_i(h))
\]
and
\[
\hat\sigma(h)=\max\{\sigma(w(h))\mid w\in W\}.
\]
\begin{Lem}\label{PropertiesOfsigma}
The function $\sigma:T\longrightarrow\mathbb R$ has the following properties.
\begin{enumerate}[(i)]
 \item $\sigma(h)=0$ holds if and only if $h\in\mathrm{Cen}(G)$.
 \item $\sigma(h)=\sigma(h^{-1})$.
 \item $\sigma(hh')\leq\sigma(h)+\sigma(h')$.
 \item $\sigma(T)=[0,\pi]$.
\end{enumerate}
The function $\hat\sigma$ has the same properties (i)--(iv). In addition, $\hat\sigma$
extends uniquely to a continuous class function on $G$.
\end{Lem}
\begin{proof}
Properties (ii) and (iii) are clear from the definition (both for $\sigma$ and for $\hat\sigma$). 
It is also clear from the definition that $\sigma$ and $\hat\sigma$ take their values in $[0,\pi]$.
Both $\sigma$ and $\hat\sigma$ map the identity element in $T$ to $0$.
Consider the Lie group homomorphism 
\[
P:T\longrightarrow\mathrm{U}(1)\times\cdots\times\mathrm{U}(1),\quad
P(h)=(\delta_1(h),\ldots,\delta_r(h)).
\]
This homomorphism has maximal rank $r$ (this can be seen from the coroots $
\delta_i^\vee$---the homomorphism $P$
surjects onto each factor $\mathrm{U}(1)$)
and is therefore surjective, because the target group is connected.
Now $\sigma(h)$ is the distance between $P(h)$ and the identity element,
with respect to the $\frac{1}{r}$-scaled $\ell_1$-product metric on the target group.
Therefore $\sigma$ assumes the value $\pi$. Since $T$ is connected, $\sigma(T)=[0,\pi]$.
It follows that $\hat\sigma(T)=[0,\pi]$, and hence claim (iv) holds.
For claim (i) we note that the kernel of $P$ is precisely the center of $G$,
see~\cite[Def.~4.38 and Prop~5.3]{Adams} or \cite[Ch.~V~Prop.~7.16]{BtD}.
Indeed, the adjoint action of $T$ on $\mathrm{Lie}(G)\otimes_{\mathbb R}\mathbb C$ is precisely given by
the characters, and the kernel of this action is the center of $G$.
Finally, the function
$\hat\sigma$ is by construction invariant under $W=\mathrm{Nor}_G(T)/T$. Therefore it extends uniquely to a 
continuous class function
on~$G$, see~\cite[Ch.~IV~Cor.~2.7]{BtD}. 
\end{proof}
\subsection{The root system, continued}
We keep the notation and the assumptions from \ref{Roots1}.
Corresponding to every simple root $\delta_i\in\Delta$ there is a connected subgroup $H_i$
of type $\mathrm{Lie}(H_i)\cong\mathfrak{su}(2)$, see~\cite[Lem.~12.2.15]{HilgertNeeb}. 
If two simple roots $\delta_i$ and $\delta_j$ are orthogonal, then they are 
\emph{strongly orthogonal}, i.e. $\delta_i\pm\delta_j\not\in\Phi$.
It follows from the commutator relations in the complexified Lie algebra
$\mathrm{Lie}(G)\otimes_{\mathbb R}\mathbb C$
that then the groups $H_i$ and $H_j$ commute, $[H_i,H_j]=1$. 
\begin{Lem}
Since $G$ is simply connected,
the groups $H_i$ are isomorphic to $\mathrm{SU}(2)$ and the root $\delta_i$ restricts
to a root on $T\cap H_i$ as described in \ref{Elementary}.
\end{Lem}
\begin{proof}
Let $\Gamma\subseteq\mathrm{Lie}(T)$ denote the subgroup generated by the coroots.
Since $G$ is simply connected, this group coincides with the kernel $I$ of the
homomorphism $\exp:\mathrm{Lie}(T)\longrightarrow T$, see~\cite[Prop.~12.4.14]{HilgertNeeb}
or~\cite[Ch.~V~Thm.~7.1]{BtD}.
By \cite[Prop~12.4.10]{HilgertNeeb}, we have for every simple root $\delta\in\Delta$
a homomorphism $\rho_\delta:\mathrm{SU}(2)\longrightarrow G$
which maps the torus elements $\left(\begin{smallmatrix}  c & \\ & \bar c \end{smallmatrix}\right)$
in $\mathrm{SU}(2)$
into $T$, such that 
$\delta(\rho_\delta(\begin{smallmatrix}  c & \\ & \bar c \end{smallmatrix}))=c^2$.
The composite $c\mapsto\rho_\delta(\begin{smallmatrix}  c & \\ & \bar c \end{smallmatrix}))$
coincides with the coroot $\delta^\vee$. 
Now $\delta^\vee(-1)=\exp(\frac{1}{2}\delta^\vee)\neq 1$, because $\frac{1}{2}\delta^\vee$ is
not in $\Gamma=I$. Therefore $\rho_\delta$ is injective.
\end{proof}

Now we get to the following key observation.
\begin{KLem}[Nikolov--Segal]\label{KLem1}
With the notation and assumptions as in \ref{Roots1}, let 
$G$ be a compact simply connected almost simple Lie group of rank $r$, with maximal
torus $T\subseteq G$. Let $\theta$ be a positive real number, and let
$h\in T$ be an element
with $\sigma(h)\geq\theta>0$. Then there exists a set $\Delta_0\subseteq\Delta$
of pairwise orthogonal simple roots such that the following conditions are
simultaneously satisfied.
\begin{enumerate}[(a)]
 \item $\Delta_0$ has at not less than $\frac{r\theta}{4}$ elements.
 \item Every element $a\in T_{\Delta_0}$ is a product of at most
 $2\lceil\frac{2\pi}{\theta}\rceil$ conjugates of $h$ and $h^{-1}$.
\end{enumerate}
\end{KLem}
\begin{proof}
We put $\Delta_1=\{\delta\in\Delta\mid\ell(\delta(h)\geq\theta/2\}$ and $t=\#\Delta_1$. Now we count
\[\textstyle
 \theta\leq \sigma(h)\leq \frac{1}{r}\big(t+(r-t)\frac{\theta}{2}\big)\leq \frac{t}{r}+\frac{\theta}{2},
\]
whence $\frac{r\theta}{2}\leq t$.

We claim that $\Delta_1$ contains a subset $\Delta_0$ consisting of pairwise orthogonal roots,
with $\frac{t}{2}\leq\#\Delta_1$. We postpone the proof of this claim, which is easy and purely
combinatorial, to \ref{WhatRemains}. We have thus $\#\Delta_0\geq\frac{r\theta}{4}$.

Now we apply Lemma~\ref{ElementaryLemma} 
to the group $H$ generated by $T$ and the $H_i$ with $\delta_i\in\Delta_0$.
We conclude that every element in $T_{\Delta_0}\subseteq H$ is
a product of at most $2\lceil\frac{2\pi}{\theta}\rceil$ conjugates of $h$ and $h^{-1}$.
%
\end{proof}
The last ingredient that we need is the following representation theoretic fact.
\begin{Lem}\label{RepLemma}
Let $\Phi$ be a reduced irreducible root system of rank $r$ in a real $r$-dimensional vector space $V$, let
$\Delta\subseteq\Phi$ be a set of simple roots and let $\Delta_0\subseteq\Delta$
be a nonempty set consisting of $s$ pairwise orthogonal roots. Let $W$ denote the Weyl group of $\Phi$.
Then there exist elements $w_1,\ldots,w_t\in W$ with the following properties.
\begin{enumerate}
 \item $w_1(\Delta_0)\cup\cdots\cup w_t(\Delta_0)$ generates $V$.
 \item $t\leq \frac{2r}{s}+3$.
\end{enumerate}
\end{Lem}
Again we postpone the proof to \ref{WhatRemains}, and state first the main result. 
\begin{Thm}[Nikolov--Segal]\label{KThm}
Let $G$ be a compact simply connected almost simple Lie group
of rank $r$. Suppose that $h\in G$ is an element with $\hat\sigma(h)\geq\theta>0$.
Then every element of $G$ can be written as a product of less than
$f(\theta)=2(\frac{8}{\theta}+3)(\frac{2\pi}{\theta}+1)$ conjugates of $h$ and $h^{-1}$.
\end{Thm}
\begin{proof}
Let $T\subseteq G$ be a maximal torus.
Every element in $G$ is conjugate to an element in $T$, see \cite[Thm.~4.21]{Adams}.
Since $\hat\sigma$ is a class function, we may assume that $h\in T$. 
After conjugating $h$ further by an element in the Weyl group $W=\mathrm{Nor}_G(T)/T$, we can assume that
$\sigma(h)=\hat\sigma(h)$.

We use the notation set up in \ref{Roots1}.
By the Key Lemma~\ref{KLem1}, there is a nonempty set $\Delta_0$ of $s$ pairwise orthogonal simple roots,
with $s\geq \frac{r\theta}{4}$. Moreover, we can generate the torus $T_{\Delta_0}$ in less than
$2(\frac{2\pi}{\theta}+1)$ steps from conjugates of $h$ and $h^{-1}$.
Let $w_1,\ldots,w_t$ be as in Lemma~\ref{RepLemma}. Then the $t$-fold product multiplication map 
$T_{\Delta_0}\times\cdots\times T_{\Delta_0}\longrightarrow T$
which maps $(a_1,\ldots,a_t)$ to $w_1(a_1)w_2(a_2)\cdots w_t(a_t)$ is a surjective homomorphism,
because it is a homomorphism of connected Lie groups whose derivative at the identity is surjective.
Hence every element in $T$ is a product of not more than $2t(\frac{2\pi}{\theta}+1)$ conjugates of
$h$ and $h^{-1}$. 
Moreover,
$t\leq \frac{2r}{s}+3$ and $\frac{1}{s}\leq\frac{4}{r\theta}$, whence $t\leq \frac{8}{\theta}+3$.
Since every element in $G$ is conjugate to an element in $T$, the claim follows.
\end{proof}
\subsection{A geometric interpretation}
The adjoint action of $G$ on itself is a prototype of a \emph{polar action}, as studied
in Riemannian geometry. The conjugacy classes in $G$ near the identity element
are the flag manifolds
for the spherical Tits building $\Delta$
associated to the complexification $G_{\mathbb C}$ of $G$.
The question that is addressed by Theorem~\ref{KThm} can be re-stated as follows.
Let $X=G/\mathrm{Cen}(G)$. Endowed with a bi-invariant Riemannian metric, this is a symmetric
space of compact type.
\begin{quote}
\emph{Let $\kappa\subseteq X$ be a nontrivial geodesic segment.
Show that every element $g\in X$ can be joined to the identity $1\in X$ by
a piecewise geodesic path, whose pieces are conjugates of $\kappa$ under the isometric action of
$G\times G$ on $X$. Bound the number of pieces needed in terms of $\kappa$.} 
\end{quote}
It would be interesting to know if there is a 'Riemannian proof' of 
Theorem~\ref{KThm}, starting from this viewpoint
of broken geodesics, using Morse Theory and polar foliations.

\medskip
In any case we now combine Theorem~\ref{KThm}, Proposition~\ref{UltraProductProp} and 
\L os' Theorem~\ref{Los}
in order to prove Theorem~\ref{N-S1}.
We need some elementary facts from the model theory of ultraproducts. Any old-fashioned
logic text
book such as \cite{BS} and \cite{CK} will do as a reference. 
See also \cite[pp.~5--7]{KW} for a very brief introduction aimed at metric geometers.
\begin{Thm}[\L os' Theorem]\label{Los}
Let $(\mathcal M_i)_{i\in I}$ be a family of $\mathcal L$-structures, for some first order language
$\mathcal L$. Let $\mu$ be a free ultrafilter on $I$ and let $^*\mathcal M$ denote the 
ultraproduct with respect to $\mu$. Let $\psi$ be a sentence in $\mathcal L$.
Then $\psi$ is true in $^*\mathcal M$ if and only if the set 
$I_\psi =\{i\in I\mid\psi\text{ is true in }\mathcal M_i\}$ is contained in $\mu$.
\end{Thm}
The proof of \L os' Theorem is not difficult and well explained e.g.~in Bell--Slomson \cite[Ch.~5, Thm.~2.1]{BS}.
\begin{Prop}
Let $(G_i)_{i\in I}$ be a family of compact simply connected almost simple Lie groups.
Let $N\trianglelefteq\prod_{i\in I}G_i$ be an abstract normal subgroup of countable index. 
Then $N=\prod_{i\in I}G_i$.
\end{Prop}
\begin{proof}
Assume to the contrary that $N\neq \prod_{i\in I}G_i$ is of countable index. 
If $pr_j(N)=K_j\neq G_j$ for some index $j\in I$,
then $K_j\subseteq G_j$ is contained in the center of $G_j$ by \ref{NormalInSemisimple} 
and hence  $G/N$ maps onto the uncountable group $G_j/K_j$, which is impossible. 
Thus $pr_j(N)=G_j$ holds for all $j\in I$. In a compact semisimple Lie group, every
element is a commutator by Goto's Theorem, 
see~\cite[Cor.~6.56]{HMCompact}. Thus $\prod_{i\in I}G_i$ is perfect and we may apply
Proposition~\ref{UltraProductProp}. It follows that there exists a free ultrafilter $\mu$
on $I$ such that \[pr_\mu(N)=K_\mu\neq \prod_{i\in I}G_i\big/\mu.\]
Let $f(\theta)=2(\frac{8}{\theta}+3)(\frac{2\pi}{\theta}+1)$.

We put $^*G=\prod_{i\in I}G_i\big/\mu$ for short. 
We choose for every $i\in I$ a maximal torus $T_i\subseteq G_i$, a set of simple roots, and
we define $\hat\sigma_i:G_i\longrightarrow[0,\pi]$ as in \ref{Roots1}.
In the ultraproduct we obtain an abelian subgroup $^*T=\prod_{i\in I}T_i\big/\mu\subseteq{}^*G$.
The ultraproduct of the $\hat\sigma_i$ is denoted by $^*\hat\sigma$. This function takes its
values in the nonstandard reals $^*\mathbb{R}$ (the ultraproduct of $\mathbb R$ with respect to $\mu$,
a real closed non-archimedean ordered field). 

Let $m$ be a positive integer. Then by Theorem~\ref{KThm},
the following first-order sentence is true in each $G_i$.
\footnote{The relevant first-order language contains predicates for elements of groups $G$, $T$, of a field $R$, 
function symbols
for maps $\hat\sigma$, $f$, function symbols $\cdot,+,\times$ for the group multiplication, the addition and 
multiplication in a field, and a relation symbol $<$ for an ordering on the field.
The number $m$ is written out as $1+1+\cdots+1$.}
\begin{quote}\em
If $g\in G$ with $\hat\sigma_i(g)\geq\frac{1}{m}$, then every element in $G_i$ is a product of at most
$f(\frac{1}{m})$ conjugates of $g$.
\end{quote}
By \L os' Theorem~\ref{Los}, the same sentence is true in the ultraproduct $^*G$. Also, every element in $^*G$ is conjugate to
an element in $^*T$ (either by \L os' Theorem, or directly, because this is true in each $G_i$ by \cite[Thm.~4.21]{Adams}).
Therefore $^*T\cap K_\mu$ consists necessarily of elements whose $\hat\sigma$-value is infinitesimally small.
Let $A\subseteq {}^*T$ denote the collection of all elements in $^*T$ whose
$^*\hat\sigma$-value is infinitesimally small. By \L os'Theorem and and the properties
of $\hat\sigma$ stated in Lemma~\ref{PropertiesOfsigma}, 
$A\subseteq{}^*T$ is a subgroup. Moreover, $|^*\hat\sigma(a)-{}^*\hat\sigma(b)|$ is
infinitesimally small if $a,b\in{}^*T$ are elements with 
$ab^{-1}\in A$. Consider the standard part
function $\mathrm{std}:{}^*\mathbb R_{\mathrm{fin}}\longrightarrow\mathbb R$
which assigns to every finite nonstandard real its real part.
The map $^*\hat\sigma$ surjects $^*T$ onto the set 
$Q=\{x\in{}^*\mathbb R\mid 0\leq x\leq\pi\}$
(again by Lemma~\ref{PropertiesOfsigma} and \L os' Theorem). Then $\mathrm{std}$ surjects
$Q$ onto $[0,\pi]$ and we have a commuting diagram
\[
\begin{tikzcd}
{}^*T \arrow{r}{{}^*\hat\sigma} \arrow{d} 
& Q \arrow{d}{\mathrm{std}} \\
{}^*T/A \arrow{r} & {[0,\pi]}.
\end{tikzcd}
\]
Therefore the quotient $^*T/A$ is uncountable.
Now $^*T/{}^*T\cap K_\mu$ injects into $^*G/K_\mu$ and 
surjects onto $^*T/A$. It follows that $^*G/K_\mu$ is uncountable. Hence
$G/N$ is also uncountable, a contradiction.
\end{proof}
\begin{proof}[Proof of Theorem~\ref{N-S1}]
Let $\phi$ denote the projection $G\longrightarrow G/N$ and put $\Gamma=G/N$
and $\Gamma_0=\phi(G^\circ)$. Then $\Gamma_0$ has finite index in $\Gamma$.

By the Approximation Theorem for compact connected groups,
there exists a compact connected abelian group
$Z$, a family $(G_i)_{i\in I}$ of simply connected compact almost simple Lie groups and a
continuous central surjective homomorphism $\psi:Z\times\prod_{i\in I}G_i\longrightarrow G^\circ$ with totally
disconnected kernel, see~\cite[Thm.~9.24]{HMCompact}. 
By Theorem~\ref{KThm}, $\phi\circ\psi$ annihilates $\prod_{i\in I}G_i$,
because $\Gamma_0$ is countable.
Since $\prod_{i\in I}G_i$ is the derived group of $Z\times\prod_{i\in I}G_i$,
the group $\Gamma_0$ is abelian, and thus $\Gamma$ is virtually abelian.

Suppose now that $\Gamma$ is residually finite.
Then $\Gamma_0$ is also residually finite.
On the other hand, $Z$ is connected and therefore divisible.
It follows from \ref{FirstFacts} that $\phi$ is constant on $Z$ and hence constant on $G^\circ$.
Thus $\phi$ is continuous and open, and $\Gamma$ is finite.
\end{proof}

\subsection{The remaining combinatorial proofs}\label{WhatRemains}
\begin{proof}[Proof of the claim made in the proof of the Key Lemma~\ref{KLem1}]
Let $\Gamma$ denote the underlying graph of the Dynkin diagram of $\Phi$. The vertices of $\Gamma$
are the simple roots in $\Delta$. This graph is a tree
and admits therefore a coloring, using red and blue, such that no two adjacent nodes have the same color.
Simple roots which are not joined by an edge are orthogonal. Thus both the red and the blue roots
form sets of pairwise orthogonal roots.
In particular we have partitioned $\Delta_1$ into two sets of pairwise orthogonal roots.
One of these sets has at least $\frac{\#\Delta_1}{2}$ elements.
\end{proof}
\begin{proof}[Proof of Lemma~\ref{RepLemma}]
Given a set $\Delta_0$ consisting of $s\geq 1$ pairwise orthogonal simple roots, 
let $t$ be the minimum number of elements $w_1,\ldots,w_t\in W$ with the property
that $w_1(\Delta_0)\cup\cdots\cup w_t(\Delta_0)$ generates $V$. The action of $W$ on $V$ is
irreducible (because $\Phi$ is irreducible). Therefore the $W$-orbit of every root $\alpha$
generates $V$. Given $\alpha$, there exist thus $r$ elements $w_1,\ldots ,w_r\in W$ such that
$w_1(\alpha),\ldots,w_r(\alpha)$ is a basis for $V$. This shows that we have always
\[
 t\leq r,
\]
and that $t=r$ if $s=1$. Now we consider the different types of root systems of
rank $r\geq 2$, and we assume that $s>1$.

For type $\mathsf A_r$ the Weyl group $W$ is the symmetric group on the set $\{1,\ldots,r+1\}$.
The reflections determined by the roots are the transpositions $(i,j)$ with $i<j$.
The reflections corresponding to the simple roots $\delta_1,\ldots,\delta_r$ are the transpositions
$(i,i+1)$, for $1\leq i\leq r$. 
The reflections corresponding to the roots in $\Delta_0$ are then pairwise commuting transpositions.
Given $1\leq k\leq r-2s-1$, there exists therefore
a permutation $w\in W$ which conjugates these $s$ commuting transpositions to
$(k,k+1),(k+2,k+3),\ldots,(k+2s-2,k+2s-1)$.
Thus $w$ maps $\Delta_0$ to the set of roots
$\eps_1\delta_{k},\eps_2\delta_{k+2},\ldots,\eps_s\delta_{k+2s-2}\subseteq\Phi$,
with $\eps_i=\pm1$. Since $V$ is spanned by $\delta_1,\ldots,\delta_r$, we may simply count 
how many subdiagrams of type $\mathsf A_{2s}$ are needed to cover the whole $\mathsf A_r$-diagram.
We obtain thus an upper bound
\[\textstyle
 t\leq 2\lceil\frac{r}{2s}\rceil<\frac{r}{s}+2<\frac{2r}{s}+3.
\]
For the types ${\mathsf B}_r,{\mathsf C}_r,{\mathsf D}_r.{\mathsf E}_r$ and $r\geq 5$ we note that there is a root subsystem
$\Phi'\subseteq\Phi$ of type $\mathsf A_{r-1}$, such that $\Delta_0$ intersects this
subsystem in a set of size at least $s-1>0$. The $r-1$-dimensional subspace $V'$
generated by $\Phi'$ is thus contained in a subspace generated by at most 
$2\lceil\frac{r-1}{2(s-1)}\rceil$ $W$-translates of $\Delta_0$. Since $V'$ is
not invariant under $W$, we can generate $V$ if we add one more $W$-translate of $\Delta_0$,
i.e.
\[\textstyle
 t\leq 2\lceil\frac{r-1}{2s-2}\rceil+1<\frac{r-1}{s-1}+3\leq\frac{2r}{s}+3.
\]
One checks directly that for $2\leq r\leq 4$ and $2\leq s<r$ one has $r\leq \frac{2r}{s}+3$,
whence $t\leq\frac{2r}{s}+3$.
\end{proof}

\subsection*{Acknowledgement}
The author thanks Martin Bridson, Arjeh Cohen, Theo Grundh\"ofer, Karl Hofmann, Karl-Hermann Neeb, and the anonymous referee for helpful remarks and clarifications.


\begin{thebibliography}{09}


\bibitem{Adams}
J. F. Adams, {\it Lectures on Lie groups}, W. A. Benjamin, Inc., New York, 1969. MR0252560 (40 \#5780)

\bibitem{Woodward}
S. Agnihotri\ and\ C. Woodward, Eigenvalues of products of unitary matrices and quantum Schubert calculus, Math. Res. Lett. {\bf 5} (1998), no.~6, 817--836. MR1671192 (2000a:14066)


\bibitem{BS}
J. L. Bell\ and\ A. B. Slomson, {\it Models and ultraproducts: An introduction}, North-Holland, Amsterdam, 1969. MR0269486 (42 \#4381)

\bibitem{BtD}
T. Br\"ocker\ and\ T. tom Dieck, {\it Representations of compact Lie groups}, translated from the German manuscript, corrected reprint of the 1985 translation, Graduate Texts in Mathematics, 98, Springer, New York, 1995. MR1410059 (97i:22005)

\bibitem{CK}
C. C. Chang\ and\ H. J. Keisler, {\it Model theory}, third edition, Studies in Logic and the Foundations of Mathematics, 73, North-Holland, Amsterdam, 1990. MR1059055 (91c:03026)

\bibitem{Dugundji}
J. Dugundji, 
{\it Topology}, 
Allyn \&\ Bacon, Boston, MA, 1966. 
MR0193606 (33 \#1824)


\bibitem{Michael}
A. A. George Michael, On normal subgroups of semisimple Lie groups, Results Math. {\bf 58} (2010), no.~1-2, 37--38. MR2672623 (2011i:22005)

\bibitem{Gluskov}
V. M. Glu\v skov, Structure of locally bicompact groups and Hilbert's fifth problem, Uspehi Mat. Nauk (N.S.) {\bf 12} (1957), no.~2 (74), 3--41. MR0101892 (21 \#698)

English translation: 
The structure of locally compact groups and Hilbert's fifth problem, 
Amer. Math. Soc. Transl. (2) {\bf 15} (1960), 55--93. MR0114872 (22 \#5690)

\bibitem{Goto}
M. Goto, On an arcwise connected subgroup of a Lie group, Proc. Amer. Math. Soc. {\bf 20} (1969), 157--162. MR0233923 (38 \#2244)

\bibitem{Guba}
V. S. Guba, A finitely generated complete group, Izv. Akad. Nauk SSSR Ser. Mat. {\bf 50} (1986), no.~5, 883--924. MR0873654 (88e:20034)

English translation: Math. USSR-Izv. {\bf 29} (1987), no.~2, 233--277.


\bibitem{HHM}
S. Hern\'andez, K. H. Hofmann\ and\ S. A. Morris, Nonmeasurable subgroups of compact groups, J. Group Theory {\bf 19} (2016), no.~1, 179--189. MR3441133

\bibitem{HNN}
G. Higman, B. H. Neumann\ and\ H. Neumann, Embedding theorems for groups, J. London Math. Soc. {\bf 24} (1949), 247--254. MR0032641 (11,322d)

\bibitem{HilgertNeeb}
J. Hilgert\ and\ K.-H. Neeb, {\it Structure and geometry of Lie groups}, Springer Monographs in Mathematics, Springer, New York, 2012. MR3025417

\bibitem{HofmannKramer}
K. H. Hofmann\ and\ L. Kramer, Transitive actions of locally compact groups on locally contractible spaces, 
J. Reine Angew. Math. {\bf 702} (2015), 227--243. MR3341471

Erratum, J. Reine Angew. Math. {\bf 702} (2015), 245--246. 

\bibitem{hofmori}
K. H. Hofmann\ and\ S. A. Morris, Transitive actions of compact groups and topological dimension, J. Algebra {\bf 234} (2000), no.~2, 454--479. MR1801101 (2002a:22006)

\bibitem{HMPro}
K. H. Hofmann\ and\ S. A. Morris, {\it The Lie theory of connected pro-Lie groups}, EMS Tracts in Mathematics, 2, European Mathematical Society (EMS), Z\"urich, 2007. MR2337107 (2008h:22001)

\bibitem{HMCompact}
K. H. Hofmann\ and\ S. A. Morris, {\it The structure of compact groups}, third edition, revised and augmented., De Gruyter Studies in Mathematics, 25, de Gruyter, Berlin, 2013. MR3114697

\bibitem{Iwasawa}
K. Iwasawa, On some types of topological groups, Ann. of Math. (2) {\bf 50} (1949), 507--558. MR0029911 (10,679a)

\bibitem{JM}
L. C. Jeffrey\ and\ A.-L. Mare, Products of conjugacy classes in $\rm SU(2)$, Canad. Math. Bull. {\bf 48} (2005), no.~1, 90--96. MR2118766 (2005m:20073)

\bibitem{Kallman}
R. R. Kallman, Every reasonably sized matrix group is a subgroup of $S\sb \infty$, Fund. Math. {\bf 164} (2000), no.~1, 35--40. MR1784652 (2001h:20005)

\bibitem{KLM}
M. Kapovich, B. Leeb\ and\ J. J. Millson, The generalized triangle inequalities in symmetric spaces and buildings with applications to algebra, Mem. Amer. Math. Soc. {\bf 192} (2008), no.~896, viii+83 pp. MR2369545 (2009d:22018)

\bibitem{KW}
L. Kramer\ and\ R. M. Weiss, Coarse equivalences of Euclidean buildings.
With an appendix by Jeroen Schillewaert and Koen Struyve, Adv. Math. {\bf 253} (2014), 1--49. MR3148544

\bibitem{MoZi}
D. Montgomery\ and\ L. Zippin, {\it Topological transformation groups}, Interscience Publishers, New York, 1955. MR0073104 (17,383b)

\bibitem{NS}
N. Nikolov\ and\ D. Segal, Generators and commutators in finite groups; abstract quotients of compact groups, Invent. Math. {\bf 190} (2012), no.~3, 513--602. MR2995181

\bibitem{Point}
F. Point, Ultraproducts and Chevalley groups, Arch. Math. Logic {\bf 38} (1999), no.~6, 355--372. MR1711404 (2001e:03070)
 
\bibitem{Ragozin}
D. L. Ragozin, A normal subgroup of a semisimple Lie group is closed, Proc. Amer. Math. Soc. {\bf 32} (1972), 632--633. MR0294563 (45 \#3633)

\bibitem{Rob}
D. J. S. Robinson, {\it A course in the theory of groups}, second edition, Graduate Texts in Mathematics, 80, Springer, New York, 1996. MR1357169 (96f:20001)

\bibitem{CPP}
H. Salzmann\ et al., {\it Compact projective planes}, de Gruyter Expositions in Mathematics, 21, de Gruyter, Berlin, 1995. MR1384300 (97b:51009)

\bibitem{ST1}
J. Saxl, S. Shelah\ and\ S. Thomas, Infinite products of finite simple groups, Trans. Amer. Math. Soc. {\bf 348} (1996), no.~11, 4611--4641. MR1376555 (97j:20031)

\bibitem{StolzThom}
A. Stolz\ and\ A. Thom, On the lattice of normal subgroups in ultraproducts of compact simple groups, Proc. Lond. Math. Soc. (3) {\bf 108} (2014), no.~1, 73--102. MR3162821

\bibitem{Stroppel}
M. Stroppel, {\it Locally compact groups}, EMS Textbooks in Mathematics, European Mathematical Society (EMS), Z\"urich, 2006. MR2226087 (2007d:22001)

\bibitem{ST2}
S. Thomas, Infinite products of finite simple groups. II, J. Group Theory {\bf 2} (1999), no.~4, 401--434. MR1718758 (2000m:20043)


\end{thebibliography}
\end{document}